\documentclass[11pt]{amsart}
\usepackage{amssymb}
\usepackage{graphicx,tikz}

\newcommand{\N}{\mathbb{N}}
\newcommand{\Z}{\mathbb{Z}}

\usepackage{float}
\usepackage{graphicx}
\usepackage{amsmath,amscd}
\usepackage{pdfpages}

\usepackage{stmaryrd}
\usepackage{amssymb}

\newtheorem{theorem}{Theorem}[section]
\newtheorem{lemma}[theorem]{Lemma}

\newtheorem{proposition}[theorem]{Proposition}
\newtheorem{claim}[theorem]{Claim}

\theoremstyle{definition}
\newtheorem{definition}[theorem]{Definition}

\newtheorem{question}[theorem]{Question}

\newtheorem{lem}[theorem]{Lemma}

\newtheorem{cor}[theorem]{Corollary}

\newtheorem{defn}[theorem]{Definition}

\author[J. Huang]{Jingyin Huang}
\address{Max Planck Institute for Mathematics
Vivatsgasse 7, 53111 Bonn, Germany}
\email{jingyin@mplm-bonn.mpg.de}

\author[M. Pawliuk]{Michael Pawliuk}
\address{Department of Mathematics and Statistics,
  University of Calgary, 612 Campus Place N.W., 2500
  University Drive NW Calgary, Alberta, Canada T2N 1N4}
\email{mpawliuk@ucalgary.ca}
\author[M. Sabok]{Marcin Sabok}
\address{Department of Mathematics and Statistics, McGill
  University, 805, Sherbrooke Street West Montreal, Quebec,
  Canada H3A 2K6}
\address{Institute of Mathematics, Polish
  Academy of Sciences, \'Sniadeckich 8, 00-655 Warszawa,
  Poland} 
\email{marcin.sabok@mcgill.ca}

\author[D. Wise]{Daniel Wise}
\address{Department of Mathematics and Statistics, McGill
  University, 805, Sherbrooke Street West Montreal, Quebec,
  Canada H3A 2K6}
\email{daniel.wise@mcgill.ca}

\thanks{This research was partially supported by the FRQNT
  (Fonds de recherche du Qu\'{e}bec) grant
  \textit{Nouveaux
    chercheurs} 2018-NC-205427 and
  by the NCN (Polish National Science Centre) through the
  grant \textit{Harmonia} no. 2015/18/M/ST1/00050}

\title[Hypertournaments and profinite topologies]{The Hrushovski
  property for hypertournaments and profinite topologies}

\begin{document}
\maketitle

\begin{abstract}
  We study the problem of extending partial isomorphisms for
  hypertournaments, which are relational structures
  generalizing tournaments. This is a generalized version
  of an old question of Herwig and Lascar. We show that the
  generalized problem has a negative answer,
  and we provide a positive answer in a special
  case. As a corollary, we show that the extension property
  holds for tournaments in case the partial isomorphisms have
  pairwise disjoint ranges and pairwise disjoint domains.
\end{abstract}

\section{Introduction}
\label{sec:introduction}

In \cite{hru} Hrushovski showed the following property for
finite graphs: whenever $G$ is a finite graph and
$\varphi_1,\ldots,\varphi_n$ are partial isomorphisms of
$G$, there exists a finite graph $G'$ containing $G$ as an
induced subgraph such that $\varphi_1,\ldots,\varphi_n$ all
extend to automorphisms of $G'$. This property appears in
the literature under various names (e.g. as the
\textit{EPPA} for \textit{Extending Property for Partial
  Automorphisms} or simply as the \textit{Hrushovski
  property}) and we say that a class $C$ of structures has
the \textit{Hrushovski property} if for any structure $M$ in
$C$ and a finite collection $\varphi_1,\ldots,\varphi_n$ of
partial isomorphisms of $M$ there exists a structure $M'$ in
$C$ which contains $M$ as a substructure and such that all
$\varphi_i$ extend to automorphisms of $M'$.
A general criterion sufficient for the Hrushovski property was given
by Herwig and Lascar \cite{hl} for structures in finite relational
languages and also by Hodkinson and Otto in
\cite{hodkinson.otto}. Recently, both these theorems were generalized
by Siniora and Solecki in \cite{siniora.solecki}. The Hrushovski
property was also studied and extended to other classes of homogeneous
structures, for instance by Solecki to the class of finite metric
spaces \cite{sol.iso} (for other proofs see \cite{vershik,
  pestov.eppa,rosendal.rz,sabok.urysohn,hkn}) or by Evans,
Hubi\v{c}ka, Kone\v{c}n\'{y} and Ne\v{s}et\v{r}il
  \cite{ehn,twographs}. For a detailed discussion of this and
  related problems the reader is advised to consult a recent survey of
  Nguyen Van Th\'e \cite{lionel} on the topic or the ICM survey
  article of Lascar \cite{lascar.icm}.

The Hrushovski property for a Fra\"iss\'e class of finite
structures is useful for the
study of automorphism groups of the corresponding
Fra\"iss\'e limits. A topological group $G$ has
\textit{ample generics} if for each $n$ there exists a dense
$G_\delta$ orbit in the diagonal action of $G$ on $G^n$
(i.e. the action
$g\cdot(g_1,\ldots,g_n)=(gg_1,\ldots,gg_n)$). Ample generics
have strong consequences, such as the automatic continuity
property of the group. Kechris and Rosendal \cite{kr} gave a general
criterion sufficient for ample generics in an automorphism group of a
Fra\"iss\'e structure that involves the Hrushovski property
for the corresponding Fra\"iss\'e class. In particular, they used the original
Hrushovski property for finite graphs in showing that the
automorphism group of the random graph has ample generics,
and consequently the automatic continuity property.

Fra\"iss\'e classes of graphs 
have been classified by Lachlan and Woodrow \cite{lw} and
for directed graphs a complete classification has been given
by Cherlin \cite{cherlin}. The Hrushovski property
for many classes of directed graphs
has been studied in the literature and proved or disproved
for many classes. In fact, the Hrushovski property implies amenability
of the automorphism group of the automorphism group of the appropriate
Fra\"iss\'e structure. The latter does not hold for the
linear tournament, the generic p.o., some weak local orders
(see e.g \cite[page 4]{pawliuk.sokic} for details).

The question for the class of tournaments is well-known and open and
appears in the Herwig and Lascar paper \cite{hl}.

\begin{question}(Herwig and Lascar \cite{hl})\label{conjecture1}
  Does the class of finite tournaments have the Hrushovski property?
\end{question}

This question is related with the problem whether the automorphism
group of the random tournament has ample generics. In fact, as proved
recently by Siniora \cite{siniora}, the two questions are
equivalent. It is worth noting that the automorphism group of the
random tournament has a comeager conjugacy class (see e.g. \cite{doucha.malicki}).

This paper we show that the Hrushovski property does
hold for tournaments in a special case when we make an extra
assumption on the partial automorphisms. In fact, this works
in a more general setting of \textit{hypertournametns} (for
definitions see Section \ref{sec:eppa}).

\begin{theorem}\label{cyclic}
  Suppose $L$ is a nontrivial set of prime numbers and $M$
  is an $L$-hypertournament. If $\varphi_1,\ldots,\varphi_n$
  are partial automorphisms of $M$ with pairwise disjoint
  domains and pairwise disjoint ranges, then there is an
  $L$-hypertournament $M'$ extending $M$ such that all
  $\varphi_i$ extend to automorphisms of $M'$.
\end{theorem}

While the above result does cover the case of tournaments
(when $L=\{2\}$), we show that the assumption on domains and
ranges of the partial automorphisms cannot be dropped
completely.

\begin{theorem}\label{negative}
   It is not true that the Hrushovski property holds for all classes of $L$-hypertournaments.
\end{theorem}

Herwig and Lascar \cite{hl} connected the question
about tournaments with a problem concerning profinite topologies on the free group. A group $G$
is called \textit{residually finite} if for every
$g\in G$ with $g\not=e$ there exists a finite-index subgroup
$H$ of $G$ such that $g\notin H$. It is well known
that the free groups are residually finite. The
\textit{profinite topology} on a residually finite group is the one
with the basis neighborhood of the identity consisiting of
finite-index subgroups. A subgroup $H$ of  a
residually finite group is \textit{separable} if it is
closed in the profinite topology and a group $G$ is
\textit{LERF} if all its finitely generated subgroups are
separable. Free groups are LERF, by a result of Hall \cite{hall}. 

Herwig and Lascar \cite{hl} found a proof of Hrushovski's theorem
using the fact that free groups are LERF and the
Ribes--Zalesskii theorem saying that products of
f.g. subgroups are closed in the profinite topology on the
free group. In a similar spirit, they showed that Question
\ref{conjecture1} is equivalent to the following Question
\ref{conjecture2}. Here, we
say that a subgroup $H$ of $G$ is \textit{closed under
  square roots} if whenever $g^2\in H$, then $g\in H$ for
any $g\in G$ and the \textit{odd adic} topology is a
refinement of the profinite topology where we take only
finite index normal subgroups of odd index as the basic
neighborhoods of the identity.

\begin{question}(Herwig and Lascar \cite{hl})\label{conjecture2}
  Is it true that for every finitely generated subgroup $H$ of the free
  group $F_n$ the following are equivalent:
  \begin{itemize}
  \item[(i)] $H$ is closed in the odd-adic topology on $F_n$,
  \item[(ii)] $H$ is closed under square roots?  
  \end{itemize}
\end{question}

The implication from (i) to (ii) above
is true and not very difficult, so the actual part of
Conjecture \ref{conjecture2} that is equivalent to
Conjecture \ref{conjecture1} is the implication from (ii) to (i).

The proof of Theorem \ref{cyclic} is based on showing that
the answer to the above question is true for cyclic
subgroups of $F_n$ even in a more general context (for
definitions see Sections \ref{sec:profinite} and
\ref{sec:eppa}).

\begin{theorem}
	\label{lem:PI}
        Let $C< F_n$ be a cyclic subgroup and $L$ be a
        nontrivial set of prime numbers. If $C$ is
        closed under $l$-roots for any $l\in L$, then $C$ is
        closed in the topology generated by the collection
        of pro-$p$ topologies with $p$ ranging over
        $L^\perp$.

\end{theorem}
 
On the other hand, the negative result in Theorem
\ref{negative} is based on the following result (for
definitions see Section \ref{sec:prelim}).

\begin{theorem}\label{counterexample}
     There exists a malnormal subgroup of $F_2$ which is not
     closed in any pro-$p$ topology.
\end{theorem}

\section{Preliminaries}\label{sec:prelim}
A \emph{graph} is a $1$-dimensional cell (CW) complex, in
which vertices are $0$-cells (points) and edges are the
$1$-cells (intervals) glued to the $0$-skeleton by their end
points. Loops and multiple edges allowed. A \emph{morphism}
of graphs is a cellular map that sends each open edge
($1$-cell without endpoints) homeomorphically onto an open
edge. A graph morphism $f:X\to X'$ is an \emph{immersion} if
it is locally injective (i.e. each point has a neighborhood
on which $f$ is injective). When $X$ and $X'$ are connected,
an immersion between them induces injective maps between
their fundamental groups.

Note that if we label the circles in the wedge of $n$ many
circles with letters $a_1,\ldots,a_n$, then we can pull back
this labelling 
to a labelling of the graph immersed to the wedge of $n$ circles. Here we use
the convention that if an edge is labelled with $a_i$, then
the reverse edge is labelled with $a_i^{-1}$. By an
\textit{immersed graph with $n$ letters} we mean a directed
graph with edges labelled with one of the $n$ letters, say
$a_1,\ldots,a_n$ such that for each vertex and $i\leq n$
there is at most one incoming edge labelled with $a_i$ and
at most one outgoing edge labelled with $a_i$. Note that
immersed graphs with $n$ letters on a vertex set $X$
correspond to the sets of $n$-many partial bijections of
$X$. Note also that an immersed graph with $n$ letters is a cover of the
bouquet of $n$ circles if and only if the partial bijections
are total. Given an immersed graph on $X$ with $n$ letters
and $x_0\in X$, the fundamental group
$\pi_1(X,x_0)$ is the subgroup of
$F_n=\langle a_1,\ldots,a_n\rangle$ consisting of those
words on $a_1,\ldots,a_n$ which form loops at $x_0$.

The following standard theorem (which follows from the work of
Stallings) shows that any finitely generated subgroup of
$F_n$ is of the above form.

\begin{theorem}
	\label{thm:stallings}
        Write $X_n$ for the wedge of $n$-circles with its
        based vertex $x_n$. For any finitely generated
        subgroup $H\le F_n$, there is a based finite graph
        $(X,x)$ and an immersion
        $$f:(X,x)\to (X_n,x_n)$$ such that $\pi_1(X,x)\cong H$
        and $f_{\ast}:\pi_1(X,x)\to \pi_1(X',x')$ is exactly
        the inclusion $H\leq F_n$.
\end{theorem}

Given a graph $X$ and a ring $R$, by its \textit{first chain
  group} $C_1(X,R)$ we mean the cellular chain group, where
each edge gives a generator in the cellular chain group. We
will always view the first homology group $H_1(X,R)$ as a subgroup of
the first chain group $C_1(X,R)$ (for details see e.g. \cite[Chapter 2.2.]{hatcher}).

Given a goup $G$ and a space $X$, we say that a covering
space $\hat{X}\to X$ is a \textit{$G$-cover} if it is a
regular cover and the deck transformation group is
isomorphic to $G$.

We will also need the following Nielsen-Schreier formula.

\begin{theorem}[\cite{may}]
  \label{thm:rank}
  Let $H$ be a subgroup of index $k$ inside a free group on
  $n$ generators. Then $H$ is a free group of rank
  $1+k(n-1)$.
\end{theorem}

\begin{defn}
  A subgroup $H$ of $G$ is \emph{malnormal} if for any
  $g\in G\setminus H$, we have
  $H\cap gHg^{-1}=\{\mathrm{1}\}$.
\end{defn}

\begin{definition}
Given a positive integer $l$, we say that a subgroup $H$ of $G$ is \textit{closed under
  $l$-roots} if for every $g\in G$  whenever $g^l\in H$, then $g\in H$.  
\end{definition}

If $l=2$, then we refer to the above by saying that $G$ is \textit{closed
under square roots}.

\begin{claim}\label{malnormalclaim}
 If $G$ is torsion-free and $H\leq G$ is malnormal, then $H$ is closed under
 $l$-roots for all $l$.
\end{claim}

\begin{proof}
  Suppose $g^l\in H$ but $g\notin H$. Then
  $H\cap g H g^{-1}=\{1\}$ by malnormality. However,
  $H\cap g H g^{-1}=\{1\}$ contains the infinite cyclic
  group generated by $g^l$, which is a contradiction.
\end{proof}

\section{Profinite topologies}\label{sec:profinite}

We consider several profinite topologies on the free group.

\begin{definition}
Given a set $P$ of prime numbers consider the topology on
the free group generated by cosets of those normal subgroups whose index is finite and
divisible only by prime numbers in $P$. We refer to this
topology as to the
\textit{pro-$P$} topology.   
\end{definition}

Note that the neighborhoods of the identity in the
pro-$P$ topology are those normal subgroups $H$ of
finite index such that all orders of elements of $F_n\slash H$ are
divisible only by numbers in $P$.

In the case $P$ consists of
a single prime number $p$, we get the pro-$p$ topology on the free group. Another
special case when $P$ consists of all odd primes was considered by Herwig and
Lascar \cite{hl} who considered the above
topology under the name of the
\textit{odd-adic} topology. For more on the pro-$P$
topologies, the reader can consult \cite{profinite.groups}.

Recall that a group $G$ is \emph{residually $p$} if the
trivial subgroup is closed in the pro-$p$ topology of $G$.
The free group $F_n$ is residually $p$ for any prime $p$
(for a short proof see e.g. \cite[Lemma 2.23]{koberda}).

\begin{definition}
Given a set $L$ of positive integers write 
$$L^\perp=\{p\ \mathrm{prime}\mid (p,l)=1\ \mathrm{for\
  any}\ l\in L\}.$$
We say that $L$ is a \textit{nontrivial} set of positive
integers if $L^\perp\not=\emptyset$.  
\end{definition}

\begin{claim}
If $H\leq F_n$ is closed in the pro-$L^\perp$ topology, then
$H$ is closed under $l$-roots for all $l\in L$.  
\end{claim}
\begin{proof}
  Indeed, if $g^l\in H$ but $g\notin H$, then $g$ cannot be
  separated from $H$ in a finite quotient of rank relatively
  prime with $l$ because for an element $\bar{g}$ of such a
  group the subgroups generated by $\bar{g}^l$ and $\bar{g}$
  are the same.
\end{proof}

\section{Fiber products of graphs}


\begin{defn}
Let $A\to X$ and $B\to X$ be maps between sets. Their \emph{fiber-product} $A\otimes_X B$ is the collection of points $(a,b)$ in $A\times B$ such that $a$ and $b$ are mapped to the same point in $X$. In the case when $A\to X$ and $B\to X$ are graph morphisms, $A\otimes_X B$ has a natural graph structure, whose vertices (resp. edges) are pairs of vertices (resp. edges) in $A,B$ that map to the same vertex (resp. edge) in $X$.  There is a commutative diagram whose arrows are graph morphisms:
	$$\begin{matrix}
	A\otimes_X B & \rightarrow & B \\
	\downarrow & & \downarrow \\
	A & \rightarrow & X \\
	\end{matrix}$$
\end{defn}

In general $A\otimes_X B$ may not be connected even if all
of $A$, $B$ and $X$ are connected (cf. Figure \ref{fig:nondiag}). If $B\to X$ is a covering
map of degree $n$, then $A\otimes_X B\to A$ is also a
covering map of the same degree. In this case it is also
called the \textit{pull-back} of the covering. The pull-back
behaves in a covariant way, i.e. for two consequtive
coverings, the pull-back of the composition is canonically
homeomorphic to the pull-back
by the second map of the pull-back by by first map (see
e.g. \cite[page 49]{steenrod}).

Note that there is a
canonical embedding from $A$ to $A\otimes_X A$, whose image
is called the \emph{diagonal component} of $A\otimes_X
A$.
Other components of $A\otimes_X A$ (if they exist) are called
\emph{non-diagonal}.

The following lemma is standard, we include a proof for the convenience of the reader.

\begin{lem}
	\label{lem:components of tensor product}
        Suppose $h_A:A\to X$ and $h_B:B\to X$ are
        immersions. Choose base points $a\in A$, $b\in B$
        such that they are mapped to the same base point
        $x\in X$.
\begin{enumerate}
\item Let $Z$ be the connected component of $A\otimes_X B$
  that contains $z=(a,b)\in A\otimes_X
  B$.
  Then $$\pi_1(Z,z)=\pi_1(A,a)\cap \pi_1(B,b)$$ (we view
  $\pi_1(Z,z),\pi_1(A,a)$ and $\pi_1(B,b)$ as subgroups of
  $\pi_1(X,x)$.)
	\item for any $g\in\pi_1(X,x)$ such that
          $$\pi_1(A,a)\cap g\pi_1(B,b)g^{-1}\not=\{1\}$$
          there is a component $C$ of
          $A\otimes_X B$ such that
          $$\pi_1(C)=\pi_1(A,a)\cap g\pi_1(B,b)g^{-1}$$ up to choices
          of base points.
\end{enumerate}
\end{lem}

\begin{proof}
  The commutativity of the diagram implies
  $\pi_1(Z,z)\subseteq\pi_1(A,a)\cap \pi_1(B,b)$. To see the
  other inclusion, let $g\in \pi_1(A,a)\cap \pi_1(B,b)$ and
  let $h:\omega\to X$ be a graph morphism from a (possibly
  subdivided) circle to $X$ representing the shortest edge
  loop in $X$ based at $x$ corresponding to $g$. We claim
  $h$ lifts to an edge loop based at $a$. To see the claim,
  let $$(\widetilde X,\tilde x)\to (X,x)$$ be the universal
  cover of $X$ with a lift $\tilde x$ of $x$. Let
  $\tilde h_A:\widetilde A\to \widetilde X$ be a base pointed
  lift of $h_A$. Since $\widetilde X$ and $\widetilde A$ are
  simply connected and $\tilde h_A$ is an immersion, $h_A$
  is an embedding and we view $\widetilde A$ as a subspace
  of $\widetilde X$. Note that $h$ lifts to an shortest edge
  path $\widetilde{\omega}\subseteq \widetilde X$ whose two
  endpoints are in $\widetilde A$. Hence
  $\widetilde{\omega}\subseteq \widetilde A$ as $\widetilde A$
  is a subtree. Now the claim follows. Similarly, we can
  also lift $h$ to an edge loop in $B$ based at $b$. This
  defines an edge loop in $A\otimes_X B$ based at $z$. Thus
  (1) follows.

  Now we prove (2). We define
  $\tilde h_B:\widetilde B\to \widetilde X$ in a way similar
  to the previous paragraph. Note that $\pi_1(A,a)$
  stabilizes $\widetilde{A}$ and $g\pi_1(B,b)g^{-1}$
  stabilizes $g\widetilde{B}$. Let
  $h\in\pi_1(A,a)\cap g\pi_1(B,b)g^{-1}$ be a non-trivial
  element. Then $h$ stabilizes a unique line
  $\ell\subset \widetilde{X}$ and acts on $\ell$ by
  translation. The uniqueness of $\ell$ implies that
  $\ell\subset \widetilde A$ and
  $\ell\subset g\widetilde B$. Thus
  $\widetilde A\cap g\widetilde B$ is non-empty and is
  stabilized by $\pi_1(A,a)\cap g\pi_1(B,b)g^{-1}$. Let
  $v\in \widetilde A\cap g\widetilde B$ be a vertex. Then
  $v$ gives rise to a pair of vertices $a'\in A$ and
  $b'\in B$ via $\widetilde A\to A$ and
  $g\widetilde B\to\widetilde B\to B$ such that $a'$ and
  $b'$ are mapped to the same vertex in $X$. Let $K$ be the
  component of $A\otimes_X B$ containing $(a',b')$. Then any
  edge path of $K$ can be lifted to an edge path inside
  $\widetilde A\cap g\widetilde B$. Thus the universal cover
  of $K$ can be identified with
  $ \widetilde A\cap g\widetilde B$ and $\pi_1 K$ can be
  identified with $\pi_1(A,a)\cap g\pi_1(B,b)g^{-1}$ up to
  a change of base points.
\end{proof}

The following is an immediate consequence of
Lemma~\ref{lem:components of tensor product}.
\begin{cor}
  \label{cor:malnormal criterion}
  Let $A\to X$ be an immersion between connected
  graphs. Then $\pi_1 A$ is malnormal in $\pi_1 X$ if and only
  if each non-diagonal component of $A\otimes_X A$ is a
  tree (i.e. simply-connected)
\end{cor}

Note that the statement of the corollary does not depend on
choices of base points in $A$ and $X$, so we omit the base
points in the statements.

Now, we record another application of the fiber product, which
together with Theorem~\ref{thm:stallings} give an algorithm
to detect whether a finitely generated subgroup of a
finitely generated free group is square free or not. 

\begin{cor}
  Let $H\le F_n$ be a finitely generated group, $X_n$ the
  wedge of $n$ circles and $f:X\to X_n$ a graph immersion
  such that $f_*(\pi_1(X))=H$. Then $H$ is square free in
  $F_n$ if and only if for any two different vertices
  $x_1,x_2\in X$, $(x_1,x_2)$ and $(x_2,x_1)$ are in
  different connected components of $X\otimes_{X_n} X$.
\end{cor}

\begin{proof}
  If $(x_1,x_2)$ and $(x_2,x_1)$ are connected by an edge
  path $\omega\subset X\otimes_{X'} X$, then $\omega$ maps
  to a loop in $X_n$ which gives a word $w\in F_n$ such that
  $w\notin H$ (since $\omega$ does not map to a loop in $X$)
  and $w^2\in H$ (since the word $w$ travels from $x_1$ to
  $x_2$, and it also travels from $x_2$ to
  $x_1$). Conversely, suppose there exists $w\in F_n$ such
  that $w\notin H$ and $w^2\in H$. Note that $w$ stabilizes
  an embedded line $\ell$ in the universal cover
  $\widetilde X_n$ of $X_n$. Since $w^2\in H$, $w^2$
  stabilizes an embedded line $\ell'$ in a lift
  $\widetilde X$ of $X$ in $\widetilde X_n$. Then
  $\ell'=\ell$. Pick a vertex
  $\tilde x_1\in\ell'\subset \widetilde X_n$ and let
  $\tilde x_2=w\tilde x_1\in \widetilde X_n$. Let $x_i$ be
  the image of $\tilde x_i$ under $\widetilde X\to X$. Then
  $x_1\neq x_2$ (since $w\notin H$) and $(x_1,x_2)$ and
  $(x_2,x_1)$ are in the same component (since $w^2\in H$).
\end{proof}

\section{Connectedness}

Another consequence of  Lemma~\ref{lem:components of tensor
  product} which will be useful is the following.

\begin{cor}\label{disconnected}
  Suppose $p:B\to X$ is a covering map of degree $d>1$ and $f:A\to X$ is an
  immersion. If $f$ lifts to an embedding $\tilde{f}:A\to B$, then the
  fiber product $A\otimes_{X}B$ is disconnected.
\end{cor}
\begin{proof}
  Note that by the definition of the fiber product, $A$ is
  contained the fiber product (identifying $A$ with its
  image under the embedding $\tilde{f}$). On the other hand,
  the fiber product is a cover of $A$, so $A$ itself has to be
  a connected component of the fiber product, which shows
  that the fiber product is disconnected when $d>1$.

\end{proof}

On the other hand, the following lemma provides a useful
condition for when the fiber product is connected.

\begin{lem}\label{spojne}
  Let $p$ be a prime and let $f:A\to X$ be an immersion
  between two connected graphs such
  that $$f_X:H_1(A,\Z\slash p\Z)\to H_1(X,\Z\slash p\Z)$$ is
  an isomorphism.  Let $X'\to X$ be a regular cover of
  degree $p$. Then $A\otimes_X X'$ is connected,
  $A\otimes_X X'\to A$ is a regular cover of degree $p$, and
  $H_1(A\otimes_X X',\mathbb Z/p\mathbb Z)$ and
  $H_1(X',\mathbb Z/p\mathbb Z)$ have the same rank.
\end{lem}

\begin{proof}
  Let $G=\pi_1(X,x)$ and let $x'\in X'$ be a lift of
  $x$. Then $G'=\pi_1(X',x')$ can be identified as the
  kernel of a surjective homomorphism
  $h:G\to \mathbb Z/p\mathbb Z$. Let $H=\pi_1(A,a)$ where
  $f(a)=x$. Then we have the following commutative
  diagram: $$\begin{matrix}
    H  &\longrightarrow & G &\longrightarrow   & \mathbb Z/p\mathbb Z \\
    &&&&\\
    \downarrow & &\downarrow   & \nearrow &  \\
    &&&&\\
    H_1(A,\mathbb Z/p\mathbb Z) &\longrightarrow   & H_1(X,\mathbb Z/p\mathbb Z) &  & \\
  \end{matrix}$$
\vspace{0.5cm}

  Note that the map $G\to H_1(A,\mathbb Z/p\mathbb Z)$
  factors as
  $G\to H_1(A,\mathbb Z)\to H_1(A,\mathbb Z/p\mathbb Z)$,
  where the first map is the abelianization map, and the
  second map is tensoring with $\mathbb Z/p\mathbb Z$. Since
  $p$ is prime, $G\to \mathbb Z/p\mathbb Z$ factors as the
  composition of two surjective homomorphisms
  $G\to H_1(X,\mathbb Z/p\mathbb Z)\to \mathbb Z/p\mathbb
  Z$.
  Since
  $H_1(A,\mathbb Z/p\mathbb Z) \to H_1(X,\mathbb Z/p\mathbb
  Z)$
  is an isomorphism, the composition
  $H\to g\to \mathbb Z/p\mathbb Z$ is surjective. Thus
  $H\cap G'$ is a normal subgroup of index $p$ in $H$. It follows that the
  connected component of $A\otimes_X X'$ containing $(a,x)$
  is a $p$-sheet regular cover of $A$, thus $A\otimes_X X'$ can not
  have other connected components.

  Since $H_1(A,\mathbb Z/p\mathbb Z)$ and
  $H_1(X,\mathbb Z/p\mathbb Z)$ are isomorphic, $\pi_1 A$
  and $\pi_1 X$ have the same rank. By
  Theorem~\ref{thm:rank} and the previous paragraph,
  $\pi_1 (A\otimes_X X')$ and $\pi_1 X'$ have the same
  rank. Thus the lemma follows.
\end{proof}

\section{Gersten's lemma}

In this section we prove a version of the Adams lemma \cite{adams}, proved
originally for $\Z$. The statement we need in Lemma
\ref{gersten} below appears implicitly in the work of
Gersten \cite{gersten} but we provide a short proof for completeness.

Let $p$ be a prime number. Given a ring $R$ with identity
$1$, write $R[t]_p$ for  $R[t]\slash(1-t^p)$. Suppose $M$ and $N$ are free
 $R$-modules and $\hat{M}$ and $\hat{N}$ are free $R[t]_p$
 modules such that $M$ and $\hat{M}$ as well as $\N$ and
 $\hat{N}$ have the same rank. Given bases $a_i$ and $\hat{a}_i$ of $M$ and
 $\hat{M}$, respectively, and $b_j$ and $\hat{b_j}$ be bases
 of $N$ and $\hat{N}$, respectively. Write $\phi_M$ for the
map induced by $a_i\mapsto \hat{a}_i$ and $t\mapsto 1$ and $\phi_N$ for the
map induced by $b_j\mapsto \hat{b}_j$ and $t\mapsto 1$.

We are going to use the following claim

\begin{claim}\label{module}
  Let $M$ and $N$ be free $(\Z\slash p\Z)$-modules and
  $\hat{M}$ and $\hat{N}$ be free $(\Z\slash
  p\Z)[t]_p$-modules such that $M$ and $\hat{M}$ as well as $N$ and
 $\hat{N}$ have the same rank and
  $\varphi_M:M\to\hat{M}$, $\varphi_N:N\to\hat{N}$ are as
  above. Suppose $f:M\to N$ is an $(\Z\slash p\Z)$-homomorphism and
  $\hat{f}:\hat{M}\to\hat{N}$ is an $(\Z\slash p\Z)[t]_p$-homomorphism
  such that the following diagram commutes

$$
  \begin{CD}
    \hat{M} @ >\hat{f} >> \hat{N}\\
    @VV \phi_M V @VV \phi_N V\\
    M @> f >> N
  \end{CD}
  $$
    If $f$
    is 1-1, then $\hat{f}$ is 1-1 too.

\end{claim}
\begin{proof}
  Let $\alpha\in \hat{M}$ be a nonzero element. Note that in
  $\Z\slash p\Z[t]_p$ we have $(1-t)^p=1-t^p=0$ and choose maximal
  $k<p$ such that $\alpha=(1-t)^k\alpha_1$ for some $\alpha_1$. Note that
  $\phi_M(\alpha_1)\not=0$ because otherwise all coordinates
  of $\alpha_1$ would be divisible by $(1-t)$, which
  contradicts maximality of $k$. 

  Now, we have $f(\phi_M(\alpha_1))\not=0$ because $f$ is
  1-1. Thus, $\phi_N(\hat{f}(\alpha_1))\not=0$ too. This
  means that at least one coordinate of $\hat{f}(\alpha_1)$
  is not divisible by $(1-t)$. Therefore, at least one
  coordinate of $(1-t)^k\hat{f}(\alpha_1)$ is not divisible
  by $(1-t)^p$ and thus is not zero. Consequently,
  $\hat{f}(\alpha)=(1-t)^k\hat{f}(\alpha_1)$ is nonzero.
\end{proof}

\begin{lemma}[Gersten]\label{gersten}
  Let $X$ and $Y$ be graphs. Let $p$ be a prime and
  $\pi_X:\hat{X}\to X$, $\pi_Y:\hat{Y}\to Y$ be
  $\mathbb{Z}\slash p\Z$-covers. Suppose $f:X\to Y$ is a
  continuous map and $\hat{f}:\hat{X}\to\hat{Y}$ is a lift
$$
  \begin{CD}
    \hat{X} @ >\hat{f} >> \hat{Y}\\
    @VV \pi_X V @VV \pi_Y V\\
    X @> f >> Y
  \end{CD}
  $$
   and the map $$f_{*}: H_1(X,\Z\slash p\Z)\to H_1(Y,\Z\slash p\Z)$$
    is 1-1. Then the map $$\hat{f}_{*}: H_1(\hat{X},\Z\slash p\Z)\to H_1(\hat{Y},\Z\slash p\Z)$$
    is 1-1 too.

\end{lemma}

\begin{proof}

  Suppose the fundamental group of $X$ isomorphic to $F_n$, and the
  finite cover $\hat X$ is corresponding to the kernel of an
  epimorphism $h:F_n \to \Z/p\Z$.  Let $\{a_i:i\leq n\}$ be the
  generators of $F_n$ and let $b_i=h(a_i)$. Since $h$ is surjective,
  at least one of the $b_i$, say $b_1$ is nontrivial. We assume
  without loss of generality that $b_1=1$. By modifying other $a_i$'s for
  $i\neq 1$, we can assume that $b_i=0$ for $i\neq 1$ (but $a_i$'s still
  form a basis of $F_n$). 

  First note that the following is a set of generators of the kernel of
  $h$:
  \[\begin{array}{ccccc}
    (a_1)^p & & & & \\
  a_2, & a_1 a_2 (a_1)^{-1} & (a_1)^2 a_2 (a_1)^{-2},  & \ldots & (a_1)^{p-1}a_2(a_1)^{1-p} \\
          & & $\ldots$ & & \\
  a_n, &  a_1 a_n(a_1)^{-1}, & (a_1)^2 a_n (a_1)^{-2} & \ldots & (a_1)^{p-1} a_n (a_1)^{1-p}
  \end{array}\]

  This gives a basis for the first homology of $\hat X$.
  Let $\hat{a_i}$ be a lift of $a_i$. Let $t$ be the action by deck
  transformation of a generator of $\Z\slash p\Z$.  This
  gives a structure of a $\Z\slash p\Z[t]_p$-module on the
  chain group of $\hat{X}$. Write $\hat{M}$ for the 
  sub-$(\Z\slash p\Z)[t]_p$-module generated by $\hat{a_i}$'s. 

   One
  checks directly that each element of the above basis of
  the homology of $\hat{X}$ is contained in
  $\hat{M}$, thus $H_1(\hat{X},\Z\slash p\Z)$ is contained in $\hat{M}$.

Note
  that $\hat{M}$ is equal
  (as a set) to the sub-$(\Z\slash p\Z)$-module generated by
  $t^k \hat{a_i}$ where $k$ is ranging between $0$ and $p-1$.

  \begin{claim}\label{free}
     $\hat{M}$ is a free $(\Z\slash p)\Z[t]_p$-module.  
  \end{claim}
  \begin{proof}
    We show that $\{t^k \hat{a_i}:0\le k\le p-1, 1\le i\le n\}$ is
    linearly independent in the first chain group of $\hat X$, viewed
    as a $\Z/p\Z$-module. It suffices to show that the
    sub-$(\Z/p\Z)$-module $E$ generated by $\{t^k\hat{a_i}\}$ has dimension
    $pn$ (since $p$ is prime, we can think of $\Z/p\Z$ vector spaces
    rather than $\Z/p\Z$-modules).

    Consider the two subspaces of $E$: $E_1$ generated by
    $\{t^k \hat{a_1}: 0\le k\le p-1\}$, and $E_2$ generated by
    $\hat{a_1}+t\hat{a_1}+...+t^{p-1}\hat{a_1}$ and
    $\{t^k \hat{a_i}:0\le k\le p-1, i\ge 2\}$. Note that $E_2$ is equal to
    the first homology of $\hat X$. Hence $\mathrm{dim}(E_2)=1+p(n-1)$
    by Theorem \ref{thm:rank}. It
    is easy to see that $\mathrm{dim}(E_1) = p$. Moreover,
    $E_1\cap E_2$ is the line spanned by
    $\hat{a_1}+t\hat{a_1}+...+t^{p-1}\hat{a_1}$. So
    $$\mathrm{dim}(E) =\mathrm{dim}(E_1)+\mathrm{dim}(E_2) -
    \mathrm{dim}(E_1\cap E_2)=pn.$$
  \end{proof}

   Now write $M$ for $H_1(X,\Z\slash p\Z)$ and let $\hat{M}$ be as
   above. Note that $M$ is a free $(\Z\slash p\Z)$-module and
   $\hat{M}$ is a free $(\Z\slash p\Z)[t]_p$-module by Claim
   \ref{free}. Let $N$ be the first chain group of $Y$, and let $\hat N$ be
   the first chain group of $\hat{Y}$ (both with coefficients $\Z\slash
   p\Z$). Note that $N$ is a $(\Z\slash p\Z)$-module and
   $\hat{N}$ is a $(\Z\slash p\Z)[t]_p$-module where $t$
   corresponds to the generator of the deck transformation
   group. $N$ is a free $(\Z\slash p\Z)$-module, by its
   definition. The basis elements of $N$ correspond to the
   edges of $Y$. The module $\hat{N}$ is free as a $(\Z\slash
   p\Z)[t]_p$-module and if for each edge $e$ of $Y$ we pick a
   lift $\hat{e}$ of $e$ to $\hat{Y}$, the the set of all
   $\hat{e}$ forms a basis of $\hat{N}$ as a $(\Z\slash
   p\Z)[t]_p$-module.

   Since all spaces considered are graphs, we treat the first homology
   group as a subgroup of the first chain group.  Note that the map
   $M\to N$ is injective, since
   $M$ is first mapped to the first homology group of
   $Y$
   (which is injective by assumption), then the first homology group
   of $Y$
   is inside the first chain group of $Y$. 

   We have a diagram induced by the continuous maps
   $$
   \begin{CD}
     \hat{M} @ >\hat{f} >> \hat{N}\\
     @VV \phi_M V @VV \phi_N V\\
     M @> f >> N
   \end{CD}
  $$

  So the map $\hat{f}$ is injective by Claim \ref{module}, hence
  we have injectivity restricted to the homology of $\hat X$, which is a
  subgroup of $\hat M$.

\end{proof}

\begin{cor}\label{isomorfizm}
  Let $p$ be a prime and let $f:A\to X$ be an immersion
  between two connected graphs such
  that $$f_X:H_1(A,\Z\slash p\Z)\to H_1(X,\Z\slash p\Z)$$ is
  an isomorphism.  Let $X'\to X$ be a regular cover of
  degree $p$ and $\hat{f}:A\otimes_X X'\to X'$ be a lift of
  $f$. 
$$
  \begin{CD}
    {A\otimes_{X} X'} @ >\hat{f} >> X'\\
    @VV  V @VV \ V\\
    A @> f >> X
  \end{CD}
  $$
Then $A\otimes_X X'$ is connected, and
  $$\hat{f}_*:H_1(A\otimes_X X',\mathbb Z/p\mathbb Z)\to
  H_1(X',\mathbb Z/p\mathbb Z)$$ is an isomorphism.
\end{cor}
\begin{proof}
  This follows directly from Lemma \ref{spojne} and Lemma
  \ref{gersten} since a 1-1 map between $\Z\slash p\Z$
  vector spaces of the same dimension must be an isomorphism.
\end{proof}

\section{Extending partial automorphisms}\label{sec:eppa}

For a natural number $l$, a sequence of $l$-tuples of
distinct elements of a set $X$ is a cycle if it is of the form
$$(x_1, x_2,\ldots,x_l), (x_2,\ldots, x_l, x_1),\ldots,(x_l,x_1,\ldots,x_{l-1}).$$

Given a group $G$ acting on a set $X$, the natural action of
$G$ on $X^l$ is the coordinatewise action. Note that for
$\bar{x}=(x_1,\ldots,x_l)\in X^l$, the $G$-orbit of $\bar{x}$ contains a
cycle if and only if there is a $g\in G$ such that $g(x_1)=x_2,\ldots,g(x_l)=x_1$.

\begin{definition}
Given a set $L$ of natural numbers, an
\textit{$L$-hypergraph} is a relational structure with one
relational symbol of arity $l$ for every $l\in L$.  
\end{definition}
Note that
a $\{2\}$-hypergraph is simply a directed graph.

Given a relational symbol $R$ of arity $l$ and a permutation
$\sigma\in\mathrm{Sym}(l)$ we write $R_\sigma$ for the
relation
$$R_\sigma(x_1,\ldots,x_l)\quad\quad\mbox{iff}\quad\quad R(x_{\sigma(1)},\ldots,x_{\sigma(l)}).$$

\begin{definition}
  Suppose $L$ is a set of natural numbers. An $L$-hypergraph
  $M$ is an \textit{$L$-hypertournament} if whenever $R$ is
  an $l$-relational symbol for $l\in L$, then for every
  tuple of distinct elements
  $\bar{x}\in M^l$ there exists a permutation
  $\sigma\in \mathrm{Sym}(\{1,\ldots,l\})$ such
  that $$M\models R_\sigma(\bar{x})$$
  and for every $R$ and permutation $\sigma$ there does not
  exists a cycle $\bar{x}_1,\ldots,\bar{x}_l$ such that
  that $$M\models R_\sigma(\bar{x}_i)$$ for
  every $i\leq l$.
\end{definition}

\noindent Note that a $\{2\}$-hypertournament is just a tournament.

The following lemma is a revamp of the analogous equivalence
proved by Herwig and Lascar for tournaments \cite{hl}.

\begin{lemma}\label{rownowaznosc}
 Let $L$ be a nontrivial set of prime numbers. The following are
 equivalent:
 \begin{itemize}
\item[(i)] The class of $L$-hypertournaments has the Hrushovski property. 
\item[(ii)] Every finitely generated subgroup of $F_n$ which is
   closed under $l$-roots for all $l\in L$ is closed in the
   pro-$L^\perp$ topology.
 
 \end{itemize}

\end{lemma}

\begin{proof}

(i)$\Rightarrow$(ii) Let $H$ be a f.g. subgroup of $F_n$ and suppose $H$ is
closed under $l$-roots for every $l\in L$. We will show that
$H$ is closed in the pro-$L^\perp$-topology. Let $a\in
F_n$ be a word which does not belong to $H$.  Consider
$X=F_n\slash H$ and note that $F_n$ acts naturally on
$X$. We introduce a structure of an
$L$-hypertournament on $X$ as follows. 

For every $l\in L$ and a
tuple of distinct elements $\bar{x}=(x_1,\ldots,x_l)\in X^l$, consider the $F_n$-orbit
of $\bar{x}$ and note that the orbit does not contain a
cycle because $H$ is closed under $l$-roots. Indeed, if
such a cycle $\bar{x}=(w_1H,\ldots, w_lH)$ existed, then for some $g\in F_n$ we would have
$g^lw_1H=w_1H$, so $w_1^{-1}g^lw_1\in H$, and thus
$w_1^{-1}gw_1\in H$, which would mean that $w_1H=w_2H$ and
contradict the assumption that $\bar{x}$ consists of
distinct elements.

This implies
that we can choose for each $l\in L$ a
collection of orbits which does not contain a cycle and
interpret $R_l$ as tuples in this collection.

Now choose a subset $X_0$ of $X$ which contains the generators of
$H$, the word $a$ and all initial subwords of the
above. Note that the generators of $F_n$ induce partial
automorphisms of $X_0$.

By our assumption there exists a finite $L$-hypertournament
$Y$ containing $X_0$. Then $G=\mathrm{Aut}(Y)$ has no elements
of order in $L$ and we get a homomorphism $\varphi$ from $F_n$ to $G$ such
that all elements of $H$ stabilize $H\in Y$ and $a$ does not
stabilize $H$. This implies that
$\varphi(a)\notin\varphi(H)$, as needed.

(ii)$\Rightarrow$(i) Now suppose $M$ is a finite $L$-hypertournament and $P$ is a
finite set of partial automorphisms of $M$. Let $k$ be the
size of $P$. Let $G_M$ be
the immersed graph induced by $P$ on $M$. We can assume that
$G_M$ is connected (extending $P$ if needed). For $x\in M$ let
$H_x=\pi_1(G_M,x)$. Recall that we treat $H_x$ as a subgroup of $F_k$.

First note that each $H_x$ is closed under $l$-roots for every
$l\in L$. Indeed, fix $x\in M$ and suppose that $w\in F_k$ is such
that $w^l\in H_x$. 
Let $x_i=w^i(x)$ with
$x_0=x$. Since $M$ is an $L$-hypertournament, there is a permutation
$\sigma$ such that $M\models (R_l)_\sigma(x_0,\ldots, x_{l-1})$. Then $w$
induces a cycle $\bar{x}_0,\ldots,\bar{x}_{l-1}$ starting with
$\bar{x}_0=(x_0,\ldots,x_{l-1})$ such that
$M\models (R_l)_\sigma(\bar{x}_i)$ for all $i< l$, which contradicts
the assumption that $M$ is an $L$-hypertournament. Thus, by our assumption, every $H_x$ is closed in the
pro-$L^\perp$ topology. 

Choose $x_0\in M$. For each $x\in M$ choose a word $w_x$
such that $w_x(x_0)=x$. For two elements $y,z\in M$ write $w_{z,y}=w_zw_y^{-1}$.

For each pair of tuples of distinct
elements $\bar{y}=(y_0,\ldots,y_{n-1})$ and
$\bar{z}=(z_0,\ldots,z_{n-1})$ in $M$ such that $M\models
R_n(\bar{y})$ but $M\not\models R_n(\bar{z})$ 
note
that $$w_{z_0,y_0}H_{y_0}\cap\ldots\cap w_{z_{n-1},y_{n-1}}H_{y_{n-1}}=\emptyset.$$

Since all $H_{y_i}$ are closed in the pro-$L^\perp$ topology
and the latter is compact Hausdorff, there exists a basic
open neighborhood of $1$, i.e. a normal subgroup $J$ whose
index is finite and not divisible by any number in $L$ such that
\begin{equation}\label{star}
w_{z_0,y_0}H_{y_0}J\cap\ldots\cap w_{z_{n-1},y_{n-1}}H_{y_{n-1}}J=\emptyset.\tag{$\ast$}
\end{equation}
for all tuples $\bar{y}$ and $\bar{z}$ as above 
and $w_x^{-1}w_y\notin H_{x_0}J$ for any $x\not=y$.

Now consider the subgroup $K=H_{x_0}J$ and let $N$ be the set
$F_k\slash K$. Since for each $x\in M$ we chose a word $w_x$
such that $w_x(x_0)=x$, we can map $x\mapsto w_xK$ and since
$K$ does not contain any $w_x^{-1}w_y$ for $x\not=y$, this
is an embedding. We need to introduce a structure of an
$L$-hypertournament on $N$ such that $M$ is a
substructure.

We have the natural action of $F_k$ on
$F_k\slash K$ and write $j:F_k\to\mathrm{Sym}(N)$. Note that
the kernel of $j$ contains $J$, so the index of the kernel
of $j$ is not divisible by any number in $L$. Write $G=F_k\slash\mathrm{ker}(j)$
and note that $G$ is a finite group without elements of
order divisible by a number in $L$, thus its rank is only
divisible by numbers in $L^\perp$.

Now we can extend the structure from $M$ to $N$ as
follows. First, for every tuple $\bar{y}=(y_0,\ldots,y_{n-1})$ in $M$
such that $M\models R(y_0,\ldots,y_{n-1})$ extend $R$ to all
tuples in the $G$-orbit of $\bar{y}$. Note that (\ref{star}) implies
that this does not change the structure on $M$. Indeed,
otherwise there is a $\bar{z}=(z_0,\ldots,z_{n-1})$ such
that $M\models\neg R(z_0,\ldots,z_{n-1})$. Suppose $g\in
F_k$ is such that $g$ maps $w_{y_i}K$ to $w_{z_i}K$, for all
$i<n$. Then $gw_{y_i}K=w_{z_i}K$ for all $i<n$. Since
$K=H_{x_0}J$ we have $w_{y_i}'\in w_{y_i}H_{x_0}$ and
$w_{z_i}'\in w_{z_i}H_{x_0}$ such that $$g\in
w_{z_i}'J(w_{y_i}')^{-1}=w_{z_i}'(w_{y_i}')^{-1}J$$ since
$J$ is normal. Now $w_{z_i}'(w_{y_i}')^{-1}\in
w_{z_i,y_i}H_{y_i}$, so we get a contradiction with (\ref{star}).

Moreover, since $G$ does not have elements of order
divisible by a number in $L$, this does not introduce cycles
of order in $L$. Next, to get a hypertournament, we extend
the structure to all remaining tuples in $N$ in the same
way, i.e. if $\bar{x}=(x_0,\ldots,x_{n-1})$ is a tuple such that
$R_n$ does not hold on any permutation of $\bar{x}$, then we
extend $R_n$ to the $G$-orbit of $\bar{x}$. Again, since $G$
does not contain elements of order in $L$, this defines a
hypertournament.

\end{proof}

\begin{definition}
  Suppose that $G$ is an immersed graph. We say that $G$ is
  a \textit{subtadpole} if $G$ has at most one vertex of
  degree $3$ at all other vertices have degree at most $2$.
\end{definition}

Note that a connected subtadpole graph looks like
a tadpole or a cycle or a tripod or a line. If a subtadpole graph is
disconnected, then at most one of its connected components
is a tadpole or a tripod and the remaining ones are cycles or lines.

Given a family $P$ of $n$ partial bijections of a set $X$, we say that
$P$ \textit{forms a subtadpole} if the corresponding immersed graph is
a subtadpole.

The following proposition is written in the spirit of the
previous lemma and gives equivalent conditions to the fact
that (ii) in the lemma holds only for cyclic groups.

\begin{proposition}\label{propozycja}
  Let $L$ be a nontrivial set of prime numbers. The following are
 equivalent:
 \begin{itemize}
\item[(i)] The class of finite $L$-hypertournaments has the
   Hrushovski property for families of partial isomorphisms
   which form subtadpoles.
 \item[(ii)] Every cyclic subgroup of $F_n$ which is
   closed under $l$-roots for all $l\in L$ is closed in the
   pro-$L^\perp$ topology.
 
\end{itemize}
\end{proposition}
\begin{proof}
  (i)$\Rightarrow$(ii) Let $C$ be a cyclic subgroup of $F_n$
  and $c\in C$
  be its generator. Assume that $C$ is closed under $l$-roots for all
  $l\in L$.  To show that $C$ is closed in the pro-$L^\perp$ topology,
  choose $a\notin C$. We may assume that $a$ does not contain $c^i$ as
  an initial subword for any $i>0$, for otherwise if $a=a'c^i$, then
  we can replace $a$ with $a'$ (note that we are reading words from right to left as we are
  talking about left actions).

  Consider first the set $X_1=F_n\slash C$ and as in the previous
  proposition introduce a structure of an $F_n$-invariant $L$-hypertournament on
  $X_1$, using the assumption that $C$ is closed under $l$-roots for
  all $l\in L$. Let $k$ be the length of $a$ and let
  $X=X_1\cup\ldots\cup X_k$ be a union of $k$ disjoint copies of
  $X_1$. Let $b$ be the longest common initial subword of
  $c$ and $a$ and let
  $a=a'b$. By our assumption on $a$, we have that $b$ does not contain
  $c$ as an initial subword. Let $x_0=bC$. For each $i<k$ let $x_i$ be the copy
  of $x_0$ in $X_i$.

  We define a family of partial bijections of $X$. First, if $c_i$ is
  the $i$-th initial subword of $c$, then let the partial bijection
  corresponding to the $i$-th letter of $c$ map $c_{i-1}C$ to $c_i C$.
  Next, for each $i<k$ if the $i$-th letter of $a'$ is $a_i$, then let
  $a_i$ map $x_i$ to $x_{i+1}$. Note that these partial bijections
  form a subtadpole on $X$. Now we define a structure of an
  $L$-hypertournament on $X$ so that all these partial bijections are
  partial isomorphisms.

  Note that the partial action of $F_n$ on $X$ induces also a partial
  action on its $l$-element subsets for each $l\in L$. In order for
  the partial maps to be partial isomorphisms, we need that the
  $L$-hypergraph structure is invariant under this partial action. Note that if an
  orbit of an $l$-tuple is entirely contained in $X_1$, then by our
  construction of the structure on $X_1$, the hypergraph relations are invariant
  on such orbits. If an orbit is not entirely contained in $X_1$ but
  did contain a cycle, then the cycle would have to be
  contained in $X_1$ (by the subtadpole structure), which is
  impossible. Thus, we can carry over the relations from $X_1$ to such orbits
  without creating cycles. Finally, all other orbits contain only
  single subsets, so we can extend the definition of the relations on
  those orbits arbitrarily not creating cycles.

  Finally, find a finite $L$-hypertournament $X'$ of $X$ which is
  closed under all subwords of $c$ and under $a$. By our assumption
  $X'$ can be extended to a finite $L$-hypertournament $X''$ such that
  all our partial isomorphisms extend to automorhpisms of $X''$. This
  defines a homomorphisms of $F_n$ to the automorphism group of
  $X''$. Since $X''$ is an $L$-hypertournament, its automorphism group
  does not contain elements of order $l$ for any $l\in L$
  and thus its rank is divisible only by numbers in
  $L^\perp$. To end the
  proof, observe that all elements of $C$ stabilize $x_0$, while $a$
  does not, so this homomorphism does not map $a$ to the
  image of $C$.

  (ii)$\Rightarrow$(i) This implication is proved exactly as
  in the previous proposition with the observation that the
  fundamental group of a subtadpole is cyclic.

\end{proof}

\section{A counterexample}

In this section we prove Theorem \ref{counterexample} and Theorem \ref{negative}.

Let $f:A\to X$ be the immersion of graphs indicated in
Figure~\ref{f:1}, where the immersion map preserves
orientation and labeling of edges. Note that the image of
$f_\ast:\pi_1 A\to \pi_1 X$ is the subgroup generated by
$aba^{-1}b^{-1}a$ and $b$ inside $F_2=\langle a,b\rangle$.

\begin{figure}[h!]
	\centering
	\includegraphics[width=1\textwidth]{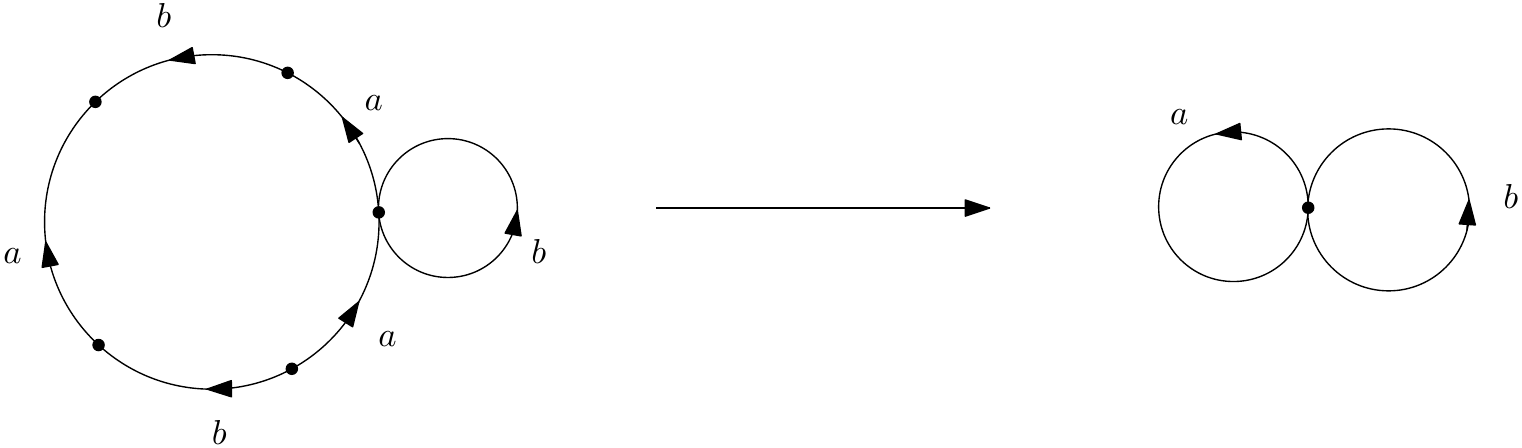}
	\caption{The immersion $f:A\to X$.}
	\label{f:1}
\end{figure}

\begin{figure}\label{fig:nondiag}
	\centering
	\includegraphics[width=1\textwidth]{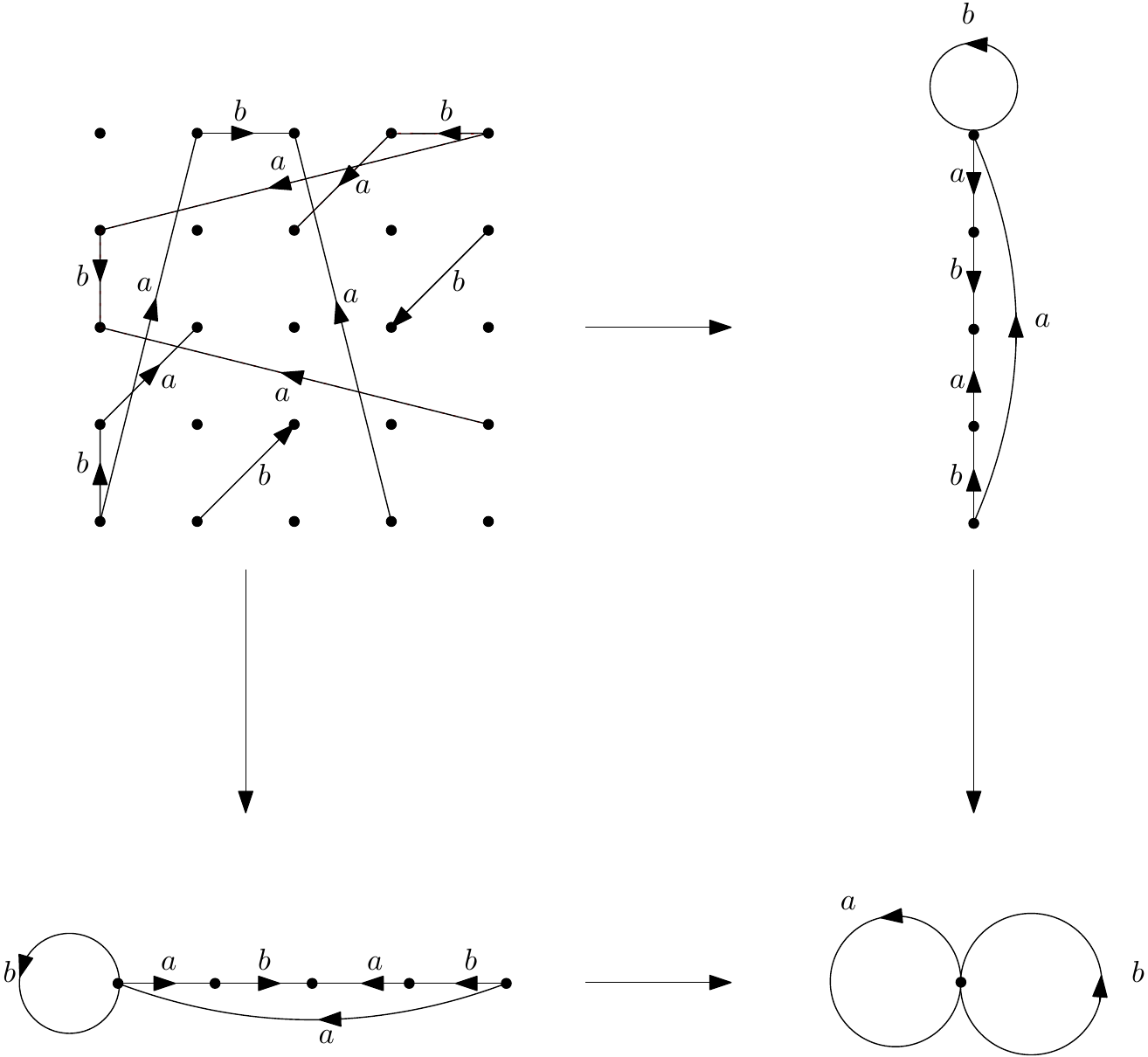}
	\caption{Only non-diagonal components of the fiber product are drawn.}
	\label{f:2}
\end{figure}

\begin{proof}[Proof of Theorem \ref{counterexample}]
  Write $G=\pi_1(A)$ for $A$ as above. 
  \begin{lem}
    $G$ is malnormal in $F_2$.
  \end{lem}

  \begin{proof}
    By Corollary~\ref{cor:malnormal criterion}, it suffices
    to show each non-diagonal component of $A \otimes_X A$
    is a tree.

    A direct computation as in Figure~\ref{f:2} implies that
    this is indeed true (note that $A$ has 5 vertices, 3
    $a$-edges and 3 $b$-edges, hence $A \otimes_X A$ has 25
    vertices, 9 $a$-edges and 9 $b$-edges, moreover, only 6
    $a$-edges and 6 $b$-edges of $A \otimes_X A$ are
    non-diagonal and they are drawn in Figure~\ref{f:2}).

  \end{proof}

  Now we show that $G$ is not closed in any pro-$p$ topology. Fix $p$
  and suppose $G$ is closed in the pro-$p$ topology. Since $G$ is
  malnormal, it is closed under $q$-th roots for all prime
  $q\not=p$, by Claim \ref{malnormalclaim},
  and we can introduce a structure of an $\{p\}^\perp$-hypertnornament
  on $A$. We do this the same way as in the proof of Lemma
  \ref{rownowaznosc}: since $A$ has 5 vertices, we need to define
  relations of arity $q=2, 3, 4, 5$. For each such $q$ and each
  $q$-tuple in $A$ choose one permutation of the tuple to include into
  the structure. Next, pass through the elements of $G$ to extend the
  structure so that $G$ acts by partial isomorphisms. The fact that
  $G$ is closed under $q$-th roots implies that the structure is an
  $\{p\}^\perp$-hypertournament. Abusing notation a bit, let us refer
  to this hypertournament also by $A$.

   If $G$ is closed in the pro-$p$ topology, then by Lemma
   \ref{rownowaznosc} (ii)$\Rightarrow$(i) (this implication uses only
   the
   assumption that the groups $H_x$ are closed in the pro-$p$ topology
   and they are all conjugates of $G$), there is a finite
   $\{p\}^\perp$-hypertournament $A'$ extending the one on
   $A$ such that both partial maps on $A$ (corresponding to $a$ and
   $b$) extend to
   isomorphisms of $A'$. This makes $A'$ a covering of
   $X$ (the wedge of two circles). Write $\pi_{A'}:A'\to X$ for the covering map. Note
   that $\pi_1(A')$ is closed under $q$-th roots for all
   prime $q\not=p$ since $A'$ is a $\{p\}^\perp$-hypertournament.

  Now, there is a further finite cover $\pi_B:B\to A'$ such that
  $B$ is a regular connected cover of $X$ (i.e. $B$ corresponds to
  the biggest normal subgroup of $F_2$ contained in the
  fundamental group of $A'$). Note since $A'$ is nontrivial,
  $B$ has degree bigger than $1$. Write $H$ for
  $\pi_1(B)$ and note that $H$ is the intersection of
  finitely many conjugates of $\pi_1(A')$, hence it is
  closed under $q$-th roots for all prime $q\not= p$. As $H$
  is normal in $F_2$, the quotient $F_2\slash H$ is a
  $p$-group and there exists a subnormal series
  $$F_2\triangleright\ldots H_1\triangleright H_0=H$$ such that each
  $H_{i+1}\slash H_i\simeq \Z\slash p\Z$. This corresponds to
  a sequence of intermediate
  subcovers $$X=X_0\leftarrow\ldots\leftarrow
  B$$ such that each $X_{i+1}\to X_i$ is a regular $\Z\slash
  p\Z$ cover. 

  Now, note that $A$ is connected and the map
  $$f_*:H_1(A,\Z\slash p\Z) \to H_1(X,\Z\slash p\Z)$$ ils an
  isomorphism because the generators of the first homology of $A$ are
  mapped to the generators of the first homology of $X$.  Inductively using Lemma \ref{spojne} and Corollary \ref{isomorfizm}, we see that each
  $A\otimes_X X_i$ is connected and the map from $H_1(A\otimes
  _X X_i,\Z\slash p\Z)$ to $H_1(X_i,\Z\slash p\Z)$ is an
  isomorphism. In particular, $A\otimes_X B$ is connected.

   On the other hand, note that since $A'$ extends $A$, the
   map $f:A\to X$ lifts to an embedding
   $f':A\to A'$. Thus, by Corollary \ref{disconnected} $A\otimes_X A'$ is
   disconnected. Thus, $(A\otimes_X A')\otimes_{A'} B$ is
   disconnected as well. But, again, by covariance of the pull-back,
   the latter is homeomorphic to $A\otimes_X B$. This
   contradicts the previous paragraph and ends the proof.

\end{proof}

\section{Cyclic subgroups}
\label{sec:cyclic-groups}

Finally, in this section we prove Theorem \ref{lem:PI} and Theorem \ref{cyclic}.

Note that every cyclic subgroup of $F_n$ is contained in a maximal cyclic
group. Indeed, if $C<F_n$ is cyclic and generated by $c$, then a maximal
cyclic subgroup of $F_n$ containing $C$ can be found by finding $d\in
F_n$ of minimal length such that $d^k=c$ for some positive integer $k$.

\begin{lemma}
	\label{lem:p}
        If $A\subset F_n$ is a maximal cyclic subgroup, then
        $A$ is closed in the pro-$p$ topology for any prime
        $p$.
\end{lemma}

\begin{proof}
  Suppose $A=\langle a\rangle$. Then the maximality of $A$
  implies that $A$ is the centralizer of $a$ in $F_n$. Pick
  $x\notin A$ and let $g=[x,a]$. By the fact $F_n$ is
  residually $p$,
  there exists a finite $p$-group $F$ and a homomorphism
  $\phi:F_n\to F$ such that $\phi(g)$ is non-trivial. Thus
  $\phi(x)$ does not commute with $\phi(a)$. Since $\phi(A)$
  is a cyclic subgroup generated by $\phi(a)$, we get that
  $\phi(x)\notin \phi(A)$.
\end{proof}

Now we prove Theorem \ref{lem:PI}.

\begin{proof}[Proof of Theorem \ref{lem:PI}]
  Let $A$ be a maximal cyclic subgroup containing
  $C$. Suppose $i=[A:C]$ and let $g\notin C$. We will
  separate $g$ from $C$.

  \emph{Case 1:} If $g\notin A$, by Lemma~\ref{lem:p}, we
  can find a $p$-group $F$ for $p\in L^\perp$ and a homomorphism
  $\phi:G\to F$ such that $\phi(g)\notin \phi(A)$,
  hence $\phi(g)\notin \phi(C)$.

  \emph{Case 2:} If $g\in A\setminus C$, we let
  $A=\langle a\rangle$ and $C=\langle a^i\rangle$. Then
  $g=a^{i j+m}$ for $j\in\mathbb Z$ and $1\le
  m<i$.
  Let $p$ be a prime factor of $i$. We claim $p\in
  L^\perp$.
  Indeed, otherwise there is $l\in L$ such that
  $p\mid l$. Suppose $l=pr_1$ and $i=pr_2$. Let
  $q=\langle i,l\rangle$ be the least common multiple of
  $i$ and $l$. Suppose $q=l r$. Now we consider the
  element $a^r$. It is clear that $(a^r)^l=a^q\in
  C$.
  However,
  $$r=\frac{q}{l}=\frac{\langle
    i,l\rangle}{l}=\frac{p\langle
    r_1,r_2\rangle}{pr_1}=\frac{\langle
    r_1,r_2\rangle}{r_1}\le r_2<i.$$
  Thus $a^r\notin C$. This contradicts that $C$ is closed
  under $l$-roots.

  Since $F_n$ is residually $p$, there is a finite $p$-group
  $F$ and a homomorphism $\phi:G\to F$ such that
  $\bar a=\phi(a)$ is non-trivial. We claim
  $\phi(g)\notin \phi(C)$. Indeed, otherwise there
  is an integer $s$ such that
  $\phi(g)=\bar a^{i j+m}=\bar a^{i s}$. Hence
  $\bar a^{i(j-s)+m}$ is trivial. Since $\bar a$ is a
  non-trivial element in a $p$-group, it has order equal to
  a power of $p$. In particular $p|i(j-s)+m$. Since
  $p|i$ by construction and $p\nmid m$, we reach a
  contradiction. Thus $\phi(g)\notin \phi(C)$
  and we are done.
\end{proof}

%



\begin{proof}[Proof of Theorem \ref{cyclic}]
  Note that if $\varphi_1,\ldots,\varphi_n$ have pairiwse
  disjoint domains and pairwise disjoint ranges, then every
  vertex in the
  immersed graph induced by them has degree at most $2$ and
  in particular is a subtadpole (in fact, it is a union of
  circles and lines). Thus, the corollary follows from
  Theorem \ref{lem:PI} and Proposition \ref{propozycja}.
\end{proof}

The case $L=\{2\}$ in the above corollary corresponds to the
case considered by Herwig and Lascar in \cite{hl}.


\bibliographystyle{plain}
\bibliography{refs}

\end{document}